\documentclass[a4paper,11pt]{article}

\usepackage{amsmath}
\usepackage{amssymb}
\usepackage{amsthm}
\usepackage{graphicx}
\usepackage{amscd}
\usepackage{epic,eepic}
\usepackage{url}
\usepackage[dvipsnames]{xcolor}
\usepackage[utf8]{inputenc} 
\usepackage{comment}
\usepackage{enumerate}   
\usepackage{graphicx}
\usepackage{epstopdf}
\usepackage{enumitem}
\usepackage{todonotes}
\usepackage{soul}

\usepackage[top=22mm, bottom=23mm, left=22mm, right=22mm]{geometry}
\usepackage[bookmarks=false,hidelinks]{hyperref}
\usepackage[capitalize]{cleveref}
\usepackage[normalem]{ulem}
\newcommand{\stkout}[1]{\ifmmode\text{\sout{\ensuremath{#1}}}\else\sout{#1}\fi}
\usepackage{tikz}
\usepackage[allcommands]{overarrows}
\usetikzlibrary{math,calc,intersections}
\usetikzlibrary{positioning,arrows,shapes,decorations.markings,decorations.pathreplacing, decorations.pathmorphing,matrix,patterns}
\tikzstyle{vertex}=[circle,draw=black,fill=black,inner sep=0,minimum size=5pt,text=white,font=\footnotesize]

\makeatletter
\renewenvironment{proof}[1][\proofname] {\par\pushQED{\qed}\normalfont\topsep6\p@\@plus6\p@\relax\trivlist\item[\hskip\labelsep\bfseries#1\@addpunct{.}]\ignorespaces}{\popQED\endtrivlist\@endpefalse}
\makeatother

\newtheorem{theorem}{\bf Theorem}[section]
\newtheorem{lemma}[theorem]{\bf Lemma}

\newtheorem{corollary}[theorem]{\bf Corollary}

\newtheorem{conjecture}[theorem]{\bf Conjecture}

\newtheorem{observation}[theorem]{\bf Observation}

\theoremstyle{definition}

\newtheorem{definition}[theorem]{\bf Definition}

\newcommand\claimproofend{\renewcommand{\qedsymbol}{$\boxdot$}
\end{proof}
\renewcommand{\qedsymbol}{$\square$}}

\title{Proof of the KAMAK tree conjecture\thanks{Department of Mathematics, ETH Z\"{u}rich, Switzerland.\\
 E-mail: {\tt $\{$micha.christoph,raphaelmario.steiner$\}$@math.ethz.ch}.
 This research was supported by the SNSF Ambizione Grant No. 216071.
}}
\author{Micha Christoph\and Raphael Steiner
}
\date{}

\begin{document}

\maketitle
\begin{abstract}
There are many intriguing questions in extremal graph theory that are well-understood in the undirected setting and yet remain elusive for digraphs. A natural instance of such a problem was recently studied by Hons, Klimo\v{s}ov\'{a}, Kucheriya, Mik\v{s}an\'{i}k, Tkadlec and Tyomkyn: What are the digraphs that have to appear as a subgraph in all digraphs of sufficiently large minimum out-degree? Hons et al.~showed that all such digraphs must be oriented forests with a specific structure, and conjectured that vice-versa all oriented forests with this specific structure appear in any digraph of sufficiently large minimum out-degree. 
In this paper, we confirm their conjecture.
\end{abstract}
\section{Introduction}
There is a large body of open problems in extremal combinatorics concerning minimum degree conditions on a graph or a digraph which enforce the existence of a specific substructure, such as a certain subgraph, a minor or a subdivision. An intriguing phenomenon which occurs for many if not most of these problems is that there seems to be a stark difference between the difficulty of undirected and directed instances, with the latter often remaining wide open or even being false. In the following, let us only discuss three selected examples which illustrate this point:

\begin{itemize}
    \item The famous and still widely open \emph{Caccetta-H\"{a}ggkvist-conjecture}~\cite{Caccetta78} from 1978 posits that the minimum out-degree threshold enforcing the existence of a directed cycle of length at most $g$ in an oriented graph with $n$ vertices equals $\left\lceil\frac{n}{g}\right\rceil$. The most popular special case of this problem that still remains open is when $g=3$, which corresponds to the minimum out-degree threshold forcing a directed triangle in oriented graphs.  The best known upper bound is of the form $(1+\varepsilon)\frac{n}{3}$ for some $\varepsilon>0$~\cite{Hladky17}, see also~\cite{Sullivan06} for a survey. In contrast, the minimum degree threshold enforcing a triangle in $n$-vertex undirected graphs is well-known to be precisely $\lfloor \frac{n}{2}\rfloor+1$ by Mantel's theorem.
    \item A classic result in extremal graph theory, due to Bollob\'{a}s-Thomason~\cite{Bollobas98} and independently Koml\'{o}s-Szemer\'{e}di~\cite{Komlos94,Komlos96} states that the minimum degree threshold enforcing a \emph{subdivision} of a clique on $t$ vertices is of order $\Theta(t^2)$. In contrast, the corresponding result turns out to be wrong for digraphs, even qualitatively. Namely, it is known that there exist digraphs of arbitrarily large minimum out-degree which do not contain an even directed cycle~\cite{Thomassen85}, and it is easy to observe that such digraphs cannot contain a subdivision of a complete digraph of order $t$ for any $t\ge 3$. Motivated by this discrepancy, Mader~\cite{Mader85} conjectured in 1985 that for every $t\in \mathbb{N}$ there exists a constant $d(t)\in \mathbb{N}$ such that every digraph of minimum out-degree at least $d(t)$ contains a subdivision of the \emph{transitively oriented} tournament of order $t$. Mader's conjecture remains widely open and has only been proved for $t\le 4$~\cite{Mader96}. Note that Mader's conjecture can be restated as saying that for every acyclic digraph $F$ there exists a constant $d(F)\in \mathbb{N}$ such that all digraphs of minimum out-degree at least $d(F)$ contain a subdivision of $F$. This statement has so far only been verified for precious few classes of digraphs $F$, such as oriented cycles~\cite{Gishboliner22} and special types of oriented trees~\cite{Aboulker19}. It remains open whether it holds for all oriented trees.
    \item A fundamental question raised by Stiebitz~\cite{Stiebitz95} in 1995 and independently by Alon~\cite{Alon96,Alon06} in 1996 and again in a survey article in 2006 asks whether for every pair $s,t$ of positive integers there exists some $d(s,t)\in \mathbb{N}$ such that every digraph $D$ with minimum out-degree $d(s,t)$ allows for a partition $A\cup B$ of its vertex-set such that $\delta^+(D[A])\geq s$ and $\delta^+(D[B])\geq t$. This question remains widely open for all pairs $(s,t)$ except $(s,t)=(1,1)$. For the latter case an affirmative answer follows from the fact that digraphs of sufficiently large minimum out-degree contain many disjoint directed cycles~\cite{Thomassen83-2, Alon96, Bucic18}, which is related to another famous conjecture in extremal digraph theory, namely the so-called \emph{Bermond-Thomassen conjecture}~\cite{BermondThomassen}. Petrova and the authors~\cite{Christoph25} recently proved that the question of Stiebitz and Alon can be reduced to the case $(s,t)=(2,2)$ in a strong sense.
    
    Also here, the analogous question for undirected graphs turns out to be much easier and was settled affirmatively by Thomassen~\cite{Thomassen83} back in 1983. The optimal bound on the necessary minimum degree in this case was later determined to be $s+t+1$ by Stiebitz~\cite{Stiebitz96}. 
\end{itemize}

In this paper, we shall be concerned with another highly natural problem of this type which was recently introduced and studied by Hons, Klimo\v{s}ov\'{a}, Kucheriya, Mik\v{s}an\'{i}k, Tkadlec and Tyomkyn~\cite{hons25}. Following their terminlogy, let us say that a (di)graph $F$ is \emph{$\delta$-enforcible}/\emph{$\delta^+$-enforcible} if there exists a constant $d(F)\in \mathbb{N}$ (depending only on $F$) such that every (di)graph with minimum (out-)degree at least $d(F)$ contains a sub(di)graph isomorphic to $F$. It is easy to observe that an undirected graph $F$ is $\delta$-enforcible if and only if it is a forest. To see this, note that a greedy embedding strategy proves that every graph with minimum degree at least $k-1$ contains every forest on at most $k$ vertices as a subgraph. In the other direction, note that the existence of graphs with arbitrarily large minimum degree and girth implies that no graph containing a cycle is $\delta$-enforcible.

One may naively expect that the situation for digraphs is analogous and that the $\delta^+$-enforcible digraphs simply consist of all oriented forests. However, as proved by Hons et al.~\cite{hons25} this is not the case: All $\delta^+$-enforcible digraphs are oriented forests with a restricted structure, namely, all their vertices of in-degree at least two must be of the same ``height''. To state their result more precisely, for a digraph $D$, let us define a \emph{height function} of $D$ as any mapping $h:V(D)\rightarrow \mathbb{Z}$ such that $h(v) = h(u)+1$ for every arc $(u,v)\in A(D)$. Note that every oriented forest admits a height function and that the latter is unique up to uniform shifts within connected components. Let us say that an oriented forest/tree is a \emph{grounded forest/tree} if it admits a height function that is constant on the set of vertices of in-degree at least $2$. See Figure~\ref{fig:grounded} for an example of such a tree. 
\begin{figure}[htb]
    \centering
\includegraphics[width=0.5\linewidth]{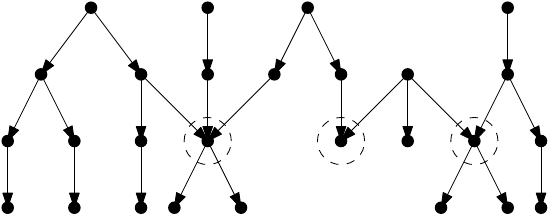}
    \caption{A grounded tree. The height function is reflected by the heights of vertices in the picture (with height increasing from top to bottom) and the vertices of in-degree at least two are circled.}
    \label{fig:grounded}
\end{figure}

We are now ready to state the aforementioned result of Hons et al.
\begin{theorem}[cf.~\cite{hons25}, Theorem~1.3]\label{thm:hons}
Every $\delta^+$-enforcible digraph is a grounded forest.
\end{theorem}
One of the main open problems posed by Hons et al.~\cite{hons25} is the conjecture that the necessary condition above for a digraph to be $\delta^+$-enforcible is also sufficient.
\begin{conjecture}[KAMAK tree conjecture, cf.~Conjecture~1.4 in~\cite{hons25}]\label{con:kamak}
Every grounded forest is $\delta^+$-enforcible.
\end{conjecture}
As the main result of this paper, we prove this conjecture. It is worth noting that Conjecture~\ref{con:kamak} reduces to the case when $F$ is connected (i.e. an oriented tree). Indeed, suppose an oriented forest $F$ has connected components $F_1,\ldots,F_k$ and $d_1,\ldots,d_k\in \mathbb{N}$ are integers such that every digraph of minimum out-degree at least $d_i$ contains a subdigraph isomorphic to $F_i$. Then, given any digraph $D$ with minimum out-degree at least, say, $\max\{d_1,\ldots,d_k\}+|V(F_1)|+\cdots+|V(F_k)|$, one can embed disjoint copies of $F_1,\ldots,F_k$ into $D$ one by one. Hence, the following result proves Conjecture~\ref{con:kamak}.
\begin{theorem}\label{thm:main}
   Every grounded tree is $\delta^+$-enforcible.
\end{theorem}
As explained in the previous discussion, this implies a characterization of $\delta^+$-enforcible digraphs.
\begin{corollary}
    A digraph $F$ is $\delta^+$-enforcible if and only if it is a grounded forest.
\end{corollary}
\paragraph*{Organization.} In the remainder of the paper, we work towards proving Theorem~\ref{thm:main}. We start in Section~\ref{sec:aux} by introducing a key structure in our proof, called \emph{$(k,d)$-broom digraphs}, and by proving several auxiliary statements that show how one can find highly structured ``sub-broom digraphs'' in a given broom digraph. Finally, in Section~\ref{sec:proof} we use these results to embed any given grounded tree (after some slight cleaning) into any given $(k,d)$-broom digraph for sufficiently large $d$, and thus, into any digraph of sufficiently high minimum out-degree, proving Theorem~\ref{thm:main}.
\paragraph*{Notation and Terminology.} For an integer $k\ge 1$ we use $[k]$ as a shorthand for the first $k$ integers. We denote by $V(D)$ and $A(D)\subseteq \{(u,v)\in V(D)^2|u\neq v\}$, respectively, the vertex and arc set of a digraph $D$. For a vertex $v\in V(D)$, we denote by $N_D^+(v), N_D^-(v)$ the out- and in-neighborhood of $v$, respectively, and by $d_D^+(v), d_D^-(v)$ its out- and in-degree in $D$. We denote by $\delta^+(D), \delta^-(D)$ the minimum out- and in-degree of $D$. Given a subset $X$ of $V(D)$ or $A(D)$, we denote by $D-X$ the digraph obtained from $D$ by deleting all elements in $X$ from $D$. Similarly, for $X\subseteq V(D)$ we denote by $D[X]$ the induced subdigraph of $D$ with vertex set $X$.
\section{Auxiliary lemmas}\label{sec:aux}
In this section, we prepare the proof of Theorem~\ref{thm:main} by proving some important auxiliary results. These results are concerned with so-called \emph{$(k,d)$-broom digraphs}, which are certain structured digraphs that will turn out to be useful for embedding grounded trees. Before being able to define a $(k,d)$-broom digraph (Definition~\ref{def:broomdigraph}) we first need to define a special type of out-arborescences that we call \emph{$(k,d)$-brooms}. Recall that an \emph{out-arborescence} is an oriented tree which has a designated root vertex such that all arcs of the tree are oriented away from the root. In the following we say that an out-arborescence is \emph{balanced} if every root-to-leaf path has the same length. The \emph{height} of an out-arborescence is the maximum length of a root-to-leaf path. 
\begin{definition}[Broom]
    Let $k,d\in \mathbb{N}$. A \emph{$(k,d)$-broom $T$ rooted at $r$} is defined to be an out-arborescence with root $r$ such that there exists an integer $1\leq \ell\leq k+1$ for which one of the following holds.
    \begin{itemize}
        \item $\ell\leq k$: $T$ is a balanced out-arborescence of height $\ell$ where all the vertices besides the leaves have out-degree $d$;
        \item $\ell=k+1$: $T$ is obtained from a balanced out-arborescence of height $k+1$ where all vertices except the leaves have out-degree $d$ by subdividing each out-arc of $r$ an arbitrary number of times.
    \end{itemize}
\end{definition}
The \emph{internal vertices} of a $(k,d)$-broom are the vertices which are neither the root nor a leaf.
\begin{definition}[Broom digraph]\label{def:broomdigraph}
    Let $D$ be a digraph and let $R\subseteq V(D)$ be non-empty. We say that $D$ is a \emph{$(k,d)$-broom digraph $D$ with root set $R$} if there are subdigraphs $(T_r)_{r\in R}$ of $D$ such that:
    \begin{itemize}
    \item $D=\bigcup_{r\in R}{T_r}$.
\item For each $r\in R$, we have that $T_r$ is a $(k,d)$-broom with root $r$ such that all leaves are in $R$ and all internal vertices are in $V(D)\setminus R$.
\item For any distinct $r_1, r_2\in R$, we have $V(T_{r_1})\cap V(T_{r_2})\subseteq R$ (that is, distinct broom digraphs are internally disjoint).
    \end{itemize}
\end{definition}
Note that by this definition, every digraph $D$ such that $d_D^+(v)=d$ for every $v\in V(D)$ is a $(k,d)$-broom digraph for any $k\ge 1$. This is because we can simply put $R:=V(D)$ and define, for each $r\in R$, the tree $T_r$ as the out-arborescence of height $1$ consisting of the root $r$ and all $d$ out-arcs of $r$ in $D$. 

After introducing these central definitions, we proceed with several auxiliary results for the proof of Theorem~\ref{thm:main}. We start with a simple observation that will be needed repeatedly later.
\begin{lemma}\label{lem: high degree}
    Let $D$ be a $(k,d)$-broom digraph with root set $R$. Let $v\in V(D)$ and suppose that there is a directed path of length at most $k$ from $v$ to $R$. Then $d_D^+(v)= d$.
\end{lemma}
\begin{proof}
    Note that we may assume that $v\in V(D)\setminus R$ as every vertex in $R$ has out-degree $d$. Let $P$ be a shortest directed path from $v$ to $R$. The internal vertices of $P$ are in $V(D)\setminus R$ as otherwise it is not the shortest directed path from $v$ to $R$. Let $r\in R$ be the root of the unique $(k,d)$-broom $T_r$ containing~$v$. Then, there exists a path $P'$ in $T_r$ from $r$ to $v$. Note that $P$ is in $T_r$ as well, since it does not have any internal vertices in $R$. Then, $P'\cup P$ is a directed path in $T_r$ from $r$ through $v$ to a leaf of $T_r$. 
    As $D$ is a $(k,d)$-broom digraph, $T_r$ is a $(k,d)$-broom. Therefore, any vertex on $P'\cup P$ with distance at most $k$ from the endpoint in $P$ has out-degree $d$ in $T_r$. 
    In particular, since $P$ has length at most $k$ and $v\in V(P)$, this implies that $d_D^+(v)=d$, as desired.
\end{proof}
Note that in order to later embed grounded trees into a given $(k,d)$-broom digraph, it will be necessary to find vertices of high in-degree (since those may exist in the grounded tree). Motivated by this,  our next goal will be to establish an auxiliary statement (namely Lemma~\ref{lem: lovasz trick} below) which shows that from any $(k,d)$-broom digraph $D$ with root set $R$ we can move to a subdigraph which is a $(k,d')$-broom digraph for a somewhat smaller but still large value of $d'$, such that all vertices in its new root set $R'\subseteq R$ have high in-degree in $D$. As one ingredient in the proof of this important lemma, we use the classic Lov\'{a}sz local lemma in the following form.
\begin{lemma}[Lov\'{a}sz Local Lemma]\label{lem: lovasz local lemma}
    Let $A_1,\ldots,A_m$ be a sequence of events such that each event occurs with probability at most $p$ and is independent of all the other events except at most $d$ of them. If $epd\leq 1$ then the event $\bigcap_{i=1}^{m}\overline{A_i}$ occurs with positive probability.
\end{lemma}
As further preparation before the proof of Lemma~\ref{lem: lovasz trick}, we need the following statement about pruning an out-arborescence to obtain a broom.
\begin{lemma}\label{lem: get broom from degrees}
    Let $k, d\in \mathbb{N}$ and let $T$ be an out-arborescence with root $r$ such that $d_T^+(r) \geq d$. Suppose that on every directed path in $T$ with vertex sequence $r=v_1,\ldots,v_h$ from $r$ to a leaf $v_h$, it holds that $d_T^+(v_i)\geq d$, for every index $i$ satisfying $\max\{1,h-k+1\}\leq i\leq h-1$. Then, $T$ contains a $(k-1,\lceil d/k\rceil)$-broom rooted at $r$ whose leaves are leaves of $T$.
\end{lemma}
\begin{proof}
    Consider a function $\phi: V(T)\to\{0,\ldots,k-1\}$ defined as follows. Let $(u_1,\ldots,u_n=r)$ be a linear ordering of $V(T)$ such that $u_i$ is a leaf of in-degree $1$ in $T[u_i,\ldots,u_n]$ for every $i\in [n-1]$. We now iteratively define $\phi(u_i)$ for $i=1,\ldots,n$ as follows: If $u_i$ is a leaf of $T$ then we define $\phi(u_i):=0$. Otherwise, we define $\phi(u_i):=\min\{\varphi(u_i)+1,k-1\}$ where $\varphi(u_i)\in \{0,\ldots,k-1\}$ is defined as the smallest index such that there exist at least $\lceil d_T^+(u_i)/k\rceil$ out-neighbors $x$ of $u_i$ with $\phi(x)=\varphi(u_i)$. This is well-defined, since by choice of the ordering $(u_1,\ldots,u_n)$, $\phi$ has already been defined on $N_T^+(u_i)$ when processing $u_i$. Let $H\subseteq T$ be the subdigraph defined by keeping only the arcs of the form $(u,v)\in V(T)$ such that $\phi(v)=\varphi(u)$. Note that by definition of $\phi$, we then have $d_H^+(u)\ge \lceil d_T^+(u)/k\rceil$ for every $u\in V(H)$.
    
    Let $T'\subseteq T$ be the out-arborescence rooted at $r$ that forms the unique connected component of $H$ containing $r$. We then have $d_{T'}^+(u)=d_{H}^+(u)\geq d_T^+(u)/k$ for every $u\in V(T')$. In particular, $u$ is a leaf of $T'$ only if it was also a leaf of $T$. Therefore, a directed path $v_1,\ldots,v_h$ on $T'$ from $r=v_1$ to a leaf $v_h$ is also a root-to-leaf path in $T$ and hence it still holds that $d_T^+(v_i)\ge d$ and thus, $d_{T'}^+(v_i)\ge \lceil d/k\rceil$, for all indices $i$ such that $\max\{1,h-k+1\}\leq i\leq h-1$. 
    
    For each $u\in V(T')$, denote by $T_u\subseteq T'$ the subtree of $T'$ rooted at $u$ (which also forms an out-arborescence). Suppose for a moment that $\varphi(u)\leq k-2$. Then, $\phi(u)=\min\{\varphi(u)+1,k-1\}=\varphi(u)+1$. By the definition of $\phi$ and the choice of $T'$, we get that every path from $u$ to a leaf in $T_u$ has length $\phi(u)$ (since the value of $\phi$ decreases by one at each step along the path).  Next, let us consider an arbitrary vertex $w\in V(T_u)$ which is not a leaf. Then, there exists a directed path $P$ in $T'$ from $r$ through $w$ to a leaf. Note that $P$ contains $u$ and, thus, also a sub-path from $u$ through $w$ to a leaf. Therefore, $w$ is among the last $\phi(u)+1\leq k$ vertices on $P$. By our observation above, this implies that $d_{T'}^+(w)\geq \lceil d/k\rceil$. Summarizing, we have shown that every non-leaf $w\in V(T_u)$ has out-degree at least $\lceil d/k\rceil$ and that $T_u$ is balanced with height $\phi(u)$. By pruning this tree, it follows that for every $u\in V(T')$ with $\varphi(u)\leq k-2$ there exists $T_u'\subseteq T_u$ which is a balanced out-arborescence of height $\phi(u)$ where all vertices which are not leaves have out-degree exactly $\lceil d/k\rceil$, and such that all leaves of $T_u'$ are also leaves of $T$. 
    
    Proceeding with the main argument, let us first consider the case that $\varphi(r)\leq k-2$. Then, $T_r'$ is a $(k-1,\lceil d/k\rceil)$-broom of height $\phi(r)$ and we are done. Next, consider the remaining case that $\varphi(r)=k-1$. Note that $d_{T'}^+(r)\ge d_T^+(r)/k\ge d/k$ by our assumption on $r$ in the lemma. Let $P_1,\ldots,P_{\lceil d/k\rceil}$ be a collection of directed paths in $T'$ from $r$ to a leaf which are vertex-disjoint besides $r$. For $1\leq i\leq \lceil d/k\rceil$, let $u_i$ be the last vertex on $P_i$ with $\phi(u_i)=k-1$. Note that this is well-defined, since $\phi$ changes by at most $1$ when moving along any arc of $T'$, and since $\phi(r)=k-1$ and $\phi(x)=0$ for every leaf $x$ of $T'$. Let $P_i'\subseteq P_i$ denote the sub-path from $r$ to $u_i$. Note that the choice of $u_i$ implies that $\varphi(u_i)=k-2$, and hence $u_i\neq r$ for every $i$. 
    Finally, consider the out-arborescence $T''$ obtained as the union of $P_i'$ as well as $T_{u_i}'$ for $1\leq i\leq \lceil d/k\rceil$. One easily checks that this is now a $(k-1,\lceil d/k\rceil)$-broom, and all its leaves are also leaves of $T'$. This concludes the proof.
\end{proof}
With Lemma~\ref{lem: lovasz local lemma} and Lemma~\ref{lem: get broom from degrees} at hand, we are now ready to formulate and prove the aforementioned result about finding a sub-broom digraph with root-set of high in-degree.
\begin{lemma}\label{lem: lovasz trick}
    Let $d,k\in \mathbb{N}$ be such that $d\ge \max\{(4k)^6,5\cdot 10^{12}\}$ and let $D$ be a $(k,d)$-broom digraph with root set $R$. Then, there exists a $(k,\lfloor d^{1/(7k)}\rfloor)$-broom digraph $D'\subseteq D$ with root set $R'\subseteq R$ such that $d^{-}_D(r)\geq d^{1/10}$ for all $r\in R'$.
\end{lemma}
\begin{proof}
    Let $U:=\{u\in V(D)|d_D^+(u)=d\}$ and let $W:=\{w\in R|d_D^-(w)<d^{1/10}\}$. Note that by definition of a $(k,d)$-broom, we have $R\subseteq U$. For each arc $(u,v)\in A(D)$ with $u\in U$, let $X_{(u,v)}$ be a Bernoulli random variable that takes value $1$ with probability $d^{-2/3}$. For arcs $(u,v)$ with $u\notin U$, we set $X_{(u,v)}=1$. Note that by definition of a $(k,d)$-broom digraph, this does not include any arc ending in a vertex in $R$. In particular, for every arc $e$ ending in $W\subseteq R$ we have that $X_e=1$ with probability $d^{-2/3}$. Let $H\subseteq D$ denote the random subdigraph obtained by keeping only the arcs $e$ for which $X_e=1$.
    
    For each $u\in U$, let $A_u$ be the random event that $d^+_{H}(u)\leq d^{1/3}/2$. A simple application of Chernoff's inequality shows that $\mathbb{P}[A_u]\le \exp(-d^{1/3}/8)$. For each $w\in W$, let $B_w$ be the random event that $d^-_H(w)\geq 2$. By a union bound over all potential choices of two incoming edges, we obtain $\mathbb{P}[B_w]\le \binom{d^{1/10}}{2}d^{-4/3}\leq d^{-17/15}$. Furthermore, the events $A_u$ depend on at most $d$ events of the form $B_w$ and are independent of all other events. The events $B_w$ depend on at most $d^{1/10}\leq d$ events of the form $A_u$ and are independent of all other events. Using our assumption on $d$, one now easily checks that $epd\le 1$ for $p:=\max\{\exp(-d^{1/3}/8), d^{-17/15}\}$ and hence we can apply Lemma~\ref{lem: lovasz local lemma}, which implies that with positive probability none of the events $A_u, u\in U$ or $B_w, w \in W$ occur. 
    
    In the following, with a slight abuse of notation, we use $H$ to denote an \emph{instance} of the random subdigraph $H$ for which none of the mentioned events occur, that is, $d_H^+(u)>d^{1/3}/2$ for all $u\in U$ and $d_H^-(w)\le 1$ for all $w\in W$. Also note, that by definition of $H$, we have $d_H^+(x)=d_D^+(x)\ge 1$ for every $x\notin U$, and hence $\delta^+(H)\ge 1$.

    Let $R'$ denote the set of vertices with in-degree at least $2$ in $H$ and note that $R'$ is non-empty. Observe that $R'\subseteq R\setminus W$. For each vertex $r\in R'$, let us denote by $U_r$ be the set of vertices reachable with a directed path starting from $r$ in the digraph $H-(R'\setminus\{r\})$. Since all the vertices in $V(H)\setminus R'$ have in-degree at most $1$, it follows that the sets $U_r, r\in R'$ are pairwise disjoint. 
    
    For each $r\in R'$, let $D_r$ be the digraph obtained from $H[U_r]$ by removing all incoming arcs of~$r$. Observe that $D_r$ is an out-arborescence with root $r$, as $r$ has in-degree $0$ and all other vertices have in-degree $1$ and are reachable from $r$. Let $L_r\subseteq U_r$ be the set of vertices $u\in V(D_r)$ with $d_{D_r}^+(u)\leq d_H^+(u)/2$ and let $S_r\subseteq D_r$ denote the out-arborescence rooted at $r$ that is induced by all the vertices reachable from $r$ via directed paths in $D_r$ that do not use an arc starting in $L_r$. Note that $V(S_r)\cap L_r$ is exactly the set of leaves of $S_r$. 

    We claim that each leaf $l$ of $S_r$ must have (in $H$) at least one out-neighbor in $R$. Indeed, suppose not. Then $d_{H}^+(l)\ge 1$ since $\delta^+(H)\ge 1$, and at the same time $d_H^+(l)=|N_H^+(l)\setminus R'|= d_{D_r}^+(l)$. Therefore, $l\notin L_r$ which contradicts the assumption that $l$ is a leaf.
    
    Next, consider any vertex $u\in V(S_r)\subseteq U_r$. Then, we either have $u\in L_r$, in which case $u$ is a leaf of $S_r$ and satisfies $$|N_H^+(u)\cap (R'\setminus \{r\})|=d_H^+(u)-d_{H-(R'\setminus \{r\})}^+(u)=d_H^+(u)-d_{H[U_r]}^+(u)\ge d_H^+(u)-d_{D_r}^+(u)-1\ge \frac{d_H^+(u)}{2}-1,$$ or we have $u\notin L_r$, in which case, by definition of $S_r$ and $L_r$, we have $d_{S_r}^+(u)=d_{D_r}^+(u)> \frac{d_{H}^+(u)}{2}$.
    
    Recall that $D$ is a $(k,d)$-broom digraph and $S_r\subseteq D$. Let $v_1,\ldots,v_h$ be the vertex-sequence of any directed path $P$ from the root $r=v_1$ to a leaf $v_h$ in $S_r$. Since $v_h$ has an out-neighbor in $R$, it follows that, for $h-k+1\leq i\leq h$, $v_i$ has a directed path (in $D$) of length at most $k$ to a vertex in $R$. By Lemma~\ref{lem: high degree}, it follows that $d_D^+(v_i)=d$ and thus $v_i\in U$ by definition. Therefore, by our choice of $H$, we have that $d_{S_r}^+(v_i)>d_H^+(v_i)/2\ge d^{1/3}/4$ for $h-k+1\leq i\leq h-1$. Furthermore, since $v_h$ is a leaf of $S_r$ we have $v_h\in L_r$ and thus, by the above, $v_h$ has at least $d_H^+(v_h)/2-1\ge d^{1/3}/4-1$ out-neighbors in $R'\setminus \{r\}$. 
    
        Observe that by the freedom of choice of $P$ in this argument, we may also conclude that indeed every leaf of $S_r$ has at least $d^{1/3}/4-1$ out-neighbors in $R'\setminus \{r\}$.

    Let $d':=\lceil d^{1/3}/4\rceil$. 
    As we have shown that every vertex in $S_r$ at distance between $1$ and $k-1$ from a leaf has out-degree at least $d'$ in $S_r$, we can now apply Lemma~\ref{lem: get broom from degrees} to the out-arborescence $S_r$. This yields the existence of a $(k-1,\lceil d'/k\rceil)$-broom that is contained in $S_r$, rooted at $r$ and such that all its leaves are also leaves of $S_r$. One can easily check that our assumption on $d$ in the lemma implies that we have $d'/k\ge d^{1/(7k)}$, and hence we can, in fact, prune the above $(k-1,\lceil d'/k\rceil)$-broom to obtain a $(k-1,\lfloor d^{1/(7k)}\rfloor)$-broom $B_r$ contained in $S_r$ with root $r$ and all leaves are also leaves of $S_r$. 
    
    Note that it follows, by definition of a $(k-1,\lfloor d^{1/(7k)}\rfloor)$-broom, that $B_r$ has at most $(d^{1/(7k)})^k=d^{1/7}$ leaves. Since furthermore, by the above, every leaf of $B_r$ has at least $d^{1/3}/4-1$ out-neighbors in $R'\setminus \{r\}$, we may extend $B_r$ to a $(k,\lfloor d^{1/(7k)}\rfloor)$-broom $T_r$ by greedily choosing a subset of $\lfloor d^{1/(7k)}\rfloor$ out-neighbors in $R'\setminus \{r\}$ disjoint from all the other already chosen vertices for each leaf of $B_r$ (in some order). Note that this is indeed possible, since at any point in the process the current leaf has at least $d^{1/3}/4-1-d^{1/(7k)}\cdot (d^{1/7}-1)\ge d^{1/7k}$ remaining leafs to choose from, where the latter inequality follows by our assumptions on $d$ in the lemma.
    
    Note that the root and leaves of $T_r$ are in $R'$ while the internal vertices are all contained in $V(H)\setminus R'$. Let $D'\subseteq H$ be the union of all the brooms $(T_r)_{r\in R'}$ defined as above. Note that $D'$ is a $(k,\lfloor d^{1/(7k)}\rfloor)$-broom digraph with root set $R'\subseteq R$. Since $R'\subseteq R\setminus W$, we have $d_D^-(r)\ge d^{1/10}$ for every $r\in R'$, as desired. This concludes the proof of the lemma.
\end{proof}

After having proven the first key auxiliary statement, let us now turn to the second one. The second main idea of the proof of Theorem~\ref{thm:main} will be to ``clean'' a given $(k,d)$-broom digraph in such a way that we can predict the positions of intersection points with the root set for any directed path up to a certain length, just based on knowing its starting point. This will later enable us to control path lengths when attempting to embed a grounded tree into a $(k,d)$-broom digraph. We start by formally fixing the properties we need in the next definition. 

\begin{definition}[Typed]
    Let $D$ be a $(k,d)$-broom digraph with root set $R$. Let further $\mathbf{a}=(\mathbf{a}_1,\ldots,\mathbf{a}_t)\in \{0,1\}^{t}$ be a binary word. We say that a vertex $v\in V(D)$ \emph{has $t$-type $\mathbf{a}$} if for every $i=1,\ldots,t$ the following holds: For every vertex $u$ in $D$ that can be reached from $v$ via a directed walk in $D$ of length exactly $i$, we have that $u\in R$ if and only if $\mathbf{a}_{i}=1$. We further say that \emph{$D$ is $t$-typed} if for every $v\in V(D)$ there exists some $\mathbf{a}\in \{0,1\}^{t}$ such that $v$ has $t$-type $\mathbf{a}$. 
    Note that, by this definition, every $(k,d)$-broom digraph is (vacuously) $0$-typed.
\end{definition}
The next statement is a simple preparatory Ramsey-type result that allows us to pass from any given out-arborescence with a coloring of its leaves to a sub-arborescence whose leaf-set is monochromatic and which preserves a good fraction of the out-degrees.
\begin{lemma}\label{lem:monochromatic}
     Let $T$ be an out-arborescence with root $r$ equipped with a coloring of its leaves with up to $C$ colors. Then, there exists an out-arborescence $T'\subseteq T$ rooted at $r$ such that the leaves of $T'$ are a monochromatic subset of the leaves of $T$ and for all $u\in V(T')$, it holds that $d_{T'}^+(u)\geq d_T^+(u)/C$.
\end{lemma}
\begin{proof}
    We perform induction on the height $\ell\ge 0$ of $T$. The statement is trivially true for $\ell=0$. Suppose then that $\ell\geq 1$ and we have already shown the statement for out-arborescences of height up to $\ell-1$. 
    
    Let $L\subseteq V(T)$ denote the set of vertices at distance $\ell-1$ from the root $r$ which are not leaves of $T$. Note that $L\neq \emptyset$, for otherwise $T$ would have height less than $\ell$. By the pigeon-hole principle, each vertex $u\in L$ has at least $d^+_{T}(u)/C$ children in some color $c$. For every $u\in L$, let us pick any such color $c$ and let us color $u$ with $c$. 
    
    Let $F\subseteq T$ be the out-arborescence of height $\ell-1$ obtained by removing all leaves of $T$ at height~$\ell$. Note that the leaf-set of $F$ is given by the union of $L$ as well as the set of leaves of $T$ at height at most $\ell-1$. Note that then, by our above color-assignment to the vertices in $L$, each of the leaves of $F$ is colored with one of $C$ colors. Therefore, we can apply induction to $F$ equipped with this coloring, obtaining a nonempty out-arborescence $F'\subseteq F$ with root $r$, such that the leaves of $F'$ form a monochromatic subset of the leaves of $F$, say, they all have color $c$, and such that $d^+_{F'}(u)\geq d^{+}_F(u)/C$ for every $u\in V(F')$. Finally, for each $u\in L\cap V(F')$, by the above definition of the colors assigned to $L$, we have that there exist at least $d^+_{T}(u)/C$ children of $u$ in $T$ which were assigned color $c$ by the original coloring of the leaves of $T$. By adding the edges from $u$ to these children of color $c$ back to $F'$, for every $u\in L\cap V(F')$, we now obtain an out-arborescence $T'\subseteq T$ of height $\ell$ such that all its leaves are $c$-colored leaves of $T$. Furthermore, we have $d_{T'}^+(u)\ge d_T^+(u)/C$ for every $u\in V(T')$ by definition of $T'$, showing that $T$ satisfies the induction hypothesis. This concludes the proof.
\end{proof}
Lemma~\ref{lem:monochromatic} directly implies the following.
\begin{corollary}\label{cor: colored trees}
    Let $T$ be a $(k,d)$-broom rooted at $v$ and suppose the leaves of $T$ are colored with $C$ colors. Then, there exists a $(k,\lceil d/C\rceil)$-broom $T'\subseteq T$ rooted at $v$ such that the leaves of $T'$ are a monochromatic subset of the leaves of $T$.
\end{corollary}
Using Corollary~\ref{cor: colored trees}, we can now prove the second key preparatory statement for Theorem~\ref{thm:main}, that allows us to pass from any given $(k,d)$-broom digraph to another broom subdigraph of slightly lower degree, but which is typed.
\begin{lemma}\label{lem: get typed}
    Let $k\in \mathbb{N}$, $t\in \{0,\ldots,k\}$ and let $D$ be a $(k,d)$-broom digraph with root set $R$. Then there exists a $t$-typed $(k,\lceil d/2^{t(t-1)/2}\rceil)$-broom digraph $D'\subseteq D$ with root set $R$.
\end{lemma}
\begin{proof}
    We prove the statement by induction on $t$. Note that any $(k,d)$-broom digraph is $0$-typed. Suppose then that $t\geq 1$ and we have already found a $(t-1)$-typed $(k,\lceil d/2^{(t-1)(t-2)/2}\rceil)$-broom digraph $D_{t-1}\subseteq D$ with root set $R$, and consisting of the union of the $(k,\lceil d/2^{(t-1)(t-2)/2}\rceil)$-brooms $(T_r)_{r\in R}$. Note that there are at most $2^{t-1}$ different $(t-1)$-types. Now, for each $r\in R$, interpret the $(t-1)$-types of the leaves of $T_r$ as a coloring with $C:=2^{t-1}$ colors. We can then apply Corollary~\ref{cor: colored trees} to $T_r$ equipped with this coloring, showing that there exists a $(k,\lceil (d/2^{(t-1)(t-2)/2})/C\rceil)=(k,\lceil (d/2^{t(t-1)/2})\rceil)$-broom $T_r'\subseteq T_r$ rooted at $r$ such that the leaves of $T_r'$ are leaves of $T_r$, all of the same $(t-1)$-type, that we denote by $\mathbf{a}_r\in \{0,1\}^{t-1}$ (for each $r\in R$). Let $D_t\subseteq D_{t-1}$ be the $(k,\lceil d/2^{t(t-1)/2}\rceil)$-broom digraph with root set $R$ defined as the union $\bigcup_{r\in R}{T_r'}$.
    
    Let $u$ be any vertex in $V(D)$, and let $r\in R$ be such that $u\in V(T_r')\setminus (R\setminus \{r\})$. Then $u$ is a non-leaf vertex of $T_r'$. Since $T_r'$ is a $(k,\lceil d/2^{t(t-1)/2}\rceil)$-broom, it follows that either there exists $\ell\ge 1$ such that every directed walk starting in $u$ of length $\ell$ ends in a leaf of $T_r'$, or every directed walk starting in $u$ of length at most $k$ does not contain any vertex in $R$. In either case, the $t$-type of $u$ is completely determined by $\mathbf{a}_r$, the $(t-1)$-type of all the leaves of $T_r'$. Since $v$ was arbitrary, this implies that $D_{t}\subseteq D$ is $t$-typed, establishing the induction hypothesis and concluding the proof.
\end{proof}

To conclude this preparatory sections, we would like to concentrate the two main results of this section (namely Lemma~\ref{lem: lovasz trick} and Lemma~\ref{lem: get typed}) into one main statement,  which can then directly be employed in the proof of Theorem~\ref{thm:main} in the next section. 
\begin{lemma}\label{Clean-up}
    Let $D$ be a $(k,d)$-broom digraph with root set $R$ such that $d\ge 10^{13k^3}$. Then there exists a $k$-typed $(k,\lceil d^{1/(8k)}\rceil)$-broom digraph $D'\subseteq D$ with root set $R'\subseteq R$ such that $d_D^-(r)\geq d^{1/10}$ for all $r\in R'$.
\end{lemma}
\begin{proof}
    By Lemma~\ref{lem: lovasz trick} there exists a $(k,\lfloor d^{1/(7k)}\rfloor)$-broom digraph $H\subseteq D$ with root set $R'\subseteq R$ such that $d_D^-(r)\geq d^{1/10}$ for all $r\in R'$. Next, we can then apply Lemma~\ref{lem: get typed} to $H$ and obtain a $k$-typed $(k,\lceil \lfloor d^{1/(7k)}\rfloor/2^{k(k-1)/2}\rceil)$-broom digraph $H'\subseteq H$ also with root-set $R'$. Using our assumption on $d$ we find that $\lfloor d^{1/(7k)}\rfloor/2^{k(k-1)/2}\geq d^{1/(8k)}$. Thus, we can further prune $H'$ to obtain a $k$-typed $(k,\lceil d^{1/(8k)}\rceil)$-broom digraph $D'\subseteq D$ with root set $R'$. 
\end{proof}
\section{Embedding the tree}\label{sec:proof}
In this section, we use our preparations from the previous section to finally prove Theorem~\ref{thm:main}. We start by setting up more notation and terminology. In the following, given any grounded tree $T$, we will use the notation $G(T):=\{x\in V(T)|d_T^-(x)\ge 2\}$ for the set of vertices of in-degree at least two and $Z(T):=\{x\in V(T)|d_T^-(x)=0\}$ for the set of sources. A useful notion we will consider in the following is that of \emph{max-grounded trees}. We say that a grounded tree $T$ is \emph{max-grounded} if for some (or, equivalently, every) height function $h:V(T)\rightarrow \mathbb{Z}$ for $T$ we have that $h(x)\ge h(y)$ for every $x\in G(T)$ and $y\in V(T)$, in other words, if the vertices in $G(T)$ are of maximum $h$-value. The reason for considering max-grounded trees is twofold: First, as we shall see next (Observation~\ref{obs:obvious} below), to prove Theorem~\ref{thm:main} it is sufficient to prove it for all max-grounded trees. Second, knowing that the tree at consideration is max-grounded comes with some convenient properties which will enable us to embed it into any $(k,d)$-broom digraph for sufficiently large $d$ and hence into all digraphs of sufficiently large minimum out-degree (Lemma~\ref{lem:keyfortheorem} below). The combination of these two insights will then prove Theorem~\ref{thm:main}.

\begin{observation}\label{obs:obvious}
If every max-grounded tree is $\delta^+$-enforcible, then so is every grounded tree, i.e., Theorem~\ref{thm:main} reduces to proving it for all max-grounded trees.
\end{observation}
\begin{proof}
    Suppose towards a contradiction that every max-grounded tree is $\delta^+$-enforcible, and yet there exists some grounded tree which is not $\delta^+$-enforcible. Let $T$ be a grounded tree of minimum order which is not $\delta^+$-enforcible. Let $h$ be a height function for $T$. Let $v$ be a vertex in $T$ maximizing $h(v)$. Since $T$ is not max-grounded, we must have $h(v)>h(x)$ for some (and hence, every) $x\in G(T)$. Note that, by definition of height functions and by choice of $v$, we must have $d_T^+(v)=0$. Since also $v\notin G(T)$, we must have $d_T^-(v)=1$, that is, $v$ is a leaf of $T$ of indegree $1$. By minimality of the order of $T$, we know that the grounded tree $T-v$ must be $\delta^+$-enforcible. Hence, there exists some constant $d(T-v)\in \mathbb{N}$ such that all digraphs of minimum out-degree at least $d(T-v)$ contain a copy of $T-v$ as a subdigraph. But then it is easy to see that every digraph of minimum out-degree at least $d(T):=\max\{d(T-v),|V(T)|\}$ contains a copy of $T$: We take a copy of $T-v$ contained in the digraph and augment it with an additional out-edge from the vertex representing the unique neighbor of $v$ in $T$ to a copy of $T$. Hence, $T$ is $\delta^+$-enforcible, a contradiction to our initial assumption on $T$. This concludes the proof. 
\end{proof}
To state the next lemma, which is the final step of the proof of Theorem~\ref{thm:main}, we show that all max-grounded trees can be embedded into $(k,d)$-brooms in a particular way. To make this precise, we use the following definition: Let $D$ be a $(k,d)$-broom digraph with root set $R$, and let $T$ be a max-grounded tree with height function $h:V(T)\rightarrow \mathbb{Z}$ such that $\max_{x\in V(T)}{h(x)}=0$. For each vertex $u\in V(D)$, let us denote by $P_D(u)$ the unique shortest path in $D$ from $R$ to $u$ (in other words, this is the root-to-leaf path of the unique $(k,d)$-broom in $D$ containing $u$ as a root or internal vertex). With this setup, we may finally define a \emph{proper copy of $T$ in $D$} as any copy $F\subseteq D$ of $T$ such that:
\begin{itemize}
    \item every vertex $x\in V(T)$ with $h(x)=0$ corresponds to a vertex of $F$ that is contained in the root set $R$ of $D$, and
    \item The paths $(P_D(x))_{x\in Z(F)}$ are pairwise vertex-disjoint and disjoint from $V(F)\setminus Z(F)$. 
\end{itemize}
Having put this important notion in place, we can now state the lemma.
\begin{lemma}\label{lem:keyfortheorem}
Let $T$ be a max-grounded tree, and let $k\ge |V(T)|$ and $d\ge 10^{13k^3\cdot (8k)^{|V(T)|}}$ be integers. Then every $(k,d)$-broom digraph contains a proper copy of $T$.
\end{lemma}
\begin{proof}
Towards a contradiction suppose the statement of the lemma is false, and let $T$ be a max-grounded tree of minimum order for which it fails. It follows that there exist integers $k\ge |V(T)|$ and $d\ge 10^{13k^3\cdot (8k)^{|V(T)|}}$ as well as a $(k,d)$-broom digraph $D$ such that $D$ contains no proper copy of $T$. Then, since every $(k,d)$-broom digraph trivially contains a proper copy of the single-vertex tree, we must have $|V(T)|\ge 2$. Let $h:V(T)\rightarrow \mathbb{Z}$ be the unique height function of $T$ such that $\max_{x\in V(T)}{h(x)}=0$. Then, by definition of a max-grounded tree, we have $h(x)=0$ for every $x\in G(T)$. Pick a leaf $l$ of $T$ minimizing $h(l)$. The choice of $l$ and $|V(T)|\ge 2$ guarantee that the set $B:=\{x\in V(T)\setminus \{l\}|h(x)=0\}$ is non-empty. Let $P$ be a shortest (not necessarily directed) path in $T$ from $l$ to a vertex in $B$, and let $v\in B$ denote the other endpoint of $P$. Note that by choice of $P$, every vertex $x\in V(P)\setminus \{v,l\}$ satisfies $h(x)<0$ and thus $x\notin G(T)$, meaning $d_T^{-}(x)\le 1$. The latter implies that there exists a vertex $s\in V(P)$ such that the two segments $P_1:=P[v,s]$ between $v$ and $s$ and $P_2:=P[s,l]$ between $s$ and $l$ of $P$ are directed paths oriented away from $s$ (note that one of these two paths could be of length $0$). Observe that $|A(P_1)|=h(v)-h(s)=-h(s)\ge h(l)-h(s)=|A(P_2)|$, so $P_2$ is at most as long as $P_1$. In the following let us denote $T':=T-l$ and note that $T'$ is a max-grounded tree smaller than $T$, and hence satisfies the statement of the lemma.

Let $D'\subseteq D$ with root set $R'\subseteq R$ be the digraph obtained by applying Lemma~\ref{Clean-up} to $D$. Then $D'$ is a $k$-typed $(k,d')$-broom digraph for $d'=\lceil d^{1/8k}\rceil \ge 10^{13k^3\cdot (8k)^{|V(T)|}/(8k)}=10^{13k^3\cdot (8k)^{|V(T')|}}$, and we have $d_D^-(r)\ge d^{1/10}$ for all $r\in R'$ as well as $\delta^-(D')\ge 1$.

Since $D'$ and $T'$ with parameters $k$ and $d'$ meet the conditions of the lemma, and since the lemma holds for $T'$, we now find that there exists a proper copy $F'$ of $T'$ in $D'$. Let $\iota:V(T')\rightarrow V(F')$ denote the embedding of $T'$ onto $F'\subseteq D'$. Note that by definition of a proper copy, and since $\max_{x\in V(T')}{h(x)}=0$, we have $\iota(x)\in R'\subseteq R$ for every $x\in V(T')$ with $h(x)=0$ and the paths $(P_{D'}(x))_{x\in Z(F')}$ are pairwise vertex-disjoint and disjoint from $V(F')\setminus Z(F')$.

We divide the remainder of the proof into three cases based on the way that $l$ attaches to $T'$. In each case, we show that $F'$ can be extended to a proper copy $F$ of $T$ in $D$, which will then contradict our assumptions on $T$. Let $u$ denote the unique neighbor of $l$ in $T$.
    \paragraph{\textbf{Case 1.}} $l$ is an in-neighbor of $u$ and $\iota(u)\notin R'$. 
    
    Since $\iota(x)\in R'$ for all $x\in V(T')$ with $h(x)=0$, it follows that $h(u)<0$ and thus $u\notin G(T)$. Consequently, $l$ is the unique in-neighbor of $u$ in $T$. Hence, $u\in Z(T')$ and $Z(T)=(Z(T')\setminus \{u\})\cup \{l\}$. As $\iota(u)\notin R'$, we get that $P_{D'}(\iota(u))$ is of length at least $1$. Let $z$ be the unique in-neighbor of $\iota(u)$ on $P_{D'}(\iota(u))$. Since $F'$ is a proper copy, we have $z\notin V(F')$. Let $F\supseteq F'$ be the copy of $T$ in $D$ obtained by embedding $l\in V(T)$ on $z\in V(D')\subseteq V(D)$. We claim that $F$ is a proper copy of $T$ in $D$. The first condition of a proper embedding is satisfied since $h(l)\le h(u)-1\le -1$ and since $F'$ was assumed to be a proper copy of $T'$. For the second condition, it remains to verify that the paths $(P_D(x))_{x\in Z(F)}$ are pairwise vertex-disjoint and disjoint from $V(F)\setminus Z(F)$. However, this easily follows since $Z(F)=(Z(F')\setminus \{\iota(u)\})\cup \{z\}$, $P_{D'}(z)\subseteq P_{D'}(\iota(u))$ by definition and $P_{D}(x)\subseteq P_{D'}(x)$ for every $x\in V(D)$ because $R'\subseteq R$. 
    \paragraph{\textbf{Case 2.}} $l$ is an in-neighbor of $u$ and $\iota(u)\in R'$.
    
    It follows that $d^-_{D}(\iota(u))\geq d^{1/10}$. Note that by definition of a broom digraph, we have that $\iota(u)$ has at most one in-neighbor in each of the $(k,d)$-brooms whose union forms $D$. As $d^{1/10}\geq k>|V(F)|-1$ by our assumptions on $k$ and $d$, this implies that we can pick an in-neighbor $z\in N^-_{D}(\iota(u))$ such that $P_{D}(z)$ is disjoint from $P_{D}(x)$, for all $x\in V(F')$. We may now extend $F'$ to a copy $F$ of $T$ in $D$ by embedding $l$ on $z$. Again, it is not hard to see that this is a proper embedding: As above we have $h(l)\le h(u)-1\le -1$, so that the first condition of properness is satisfied since $F'$ is a proper copy of $T'$ in $D'$. The second condition follows since $Z(F)=(Z(F')\setminus \{\iota(u)\})\cup \{z\}$, the paths $P_{D}(x)\subseteq P_{D'}(x)$ for $x\in Z(F')\setminus \{\iota(u)\}$ are pairwise disjoint from each other and from $(V(F')\setminus Z(F'))\cup\{\iota(u)\}\supseteq V(F)\setminus Z(F)$ and since by our choice of $z$ we have that $P_D(z)$ is disjoint from all the above paths as well as from $V(F')=V(F)\setminus \{z\}$. 
    \paragraph{\textbf{Case 3.}} $l$ is an out-neighbor of $u$. 
    
    Let $\ell_1,\ell_2$ denote the lengths of the paths $P_1,P_2$, respectively, and note/recall that $1\le \ell_2\leq \ell_1< |V(T)|\le k$. Let $Q_1:=\iota(P_1)$ be the image of $P_1$ in the copy $F'$ of $T$. Since $F'$ is a proper copy of $T'$ in $D'$, the endpoint of $Q_1$, namely $\iota(v)$, must be contained in $R'$ (this is because $v\in B$ and hence $h(v)=0$ by our choice of $v$). Hence, $Q_1$ is a directed path of length $\ell_1\in [k]$ from $\iota(s)$ to a vertex in $R'$. Since $D'$ is $k$-typed, it follows that \emph{every} directed walk in $D'$ of length $\ell_1$ starting in $\iota(s)$ must end in $R'$. Let $Q_2:=\iota(P_2-l)$ denote the image of $P_2-l$ in $F'$ with endpoints $\iota(s)$ and $\iota(u)$. We next claim that there exists a directed path of length at most $k$ in $D'$ from $\iota(u)$ to $R'$. To see this, let $Q$ be a shortest directed path in $D$ from $\iota(u)$ to $R'$ (note that such a path exists by definition of a broom digraph). Towards a contradiction, suppose that $|A(Q)|>k$ and consider the directed walk $W$ in $D'$ starting in $\iota(s)$ which is obtained by concatenating $Q_2$ and $Q$. By assumption, the length of $W$ is at least $|A(Q)|>k$. Then by the above we must have that the $(\ell_1+1)$-st vertex we meet when traversing $W$ starting from $\iota(s)$ must be contained in $R'$. Since further the initial segment $Q_2$ of $W$ has length $|A(Q_2)|=\ell_2-1<\ell_1$, it follows that $x$ must be a vertex of $Q$ distinct from its endpoint. This, however, is a contradiction, since $x\in R'$ and $Q$ was a shortest directed path from $\iota(u)$ to $R'$. This shows that our above assumption was wrong, there indeed exists a directed path in $D'$ of length at most $k$ from $\iota(u)$ to $R'$.
    
    Hence, we may apply Lemma~\ref{lem: high degree} to the vertex $\iota(u)$ in $D'$, which yields that $d_{D'}^+(\iota(u))=d'\gg k$. Observe that by definition of a $(k,d)$-broom digraph, for each $x\in V(F')$, at most one out-neighbor of $\iota(u)$ in $D'$ is contained in $P_{D'}(x)$. Thus, we may choose a vertex $z\in N^+_{D'}(\iota(u))$ which does not appear on any of the paths $P_{D'}(x)$, $x\in V(F')$. Now let $F$ be the copy of $T$ in $D'\subseteq D$ obtained by embedding $l$ at $z$. We claim that this is a proper copy of $F$ in $D$: Indeed, we have $Z(F)=Z(F')$ and hence the paths $P_D(x)\subseteq P_{D'}(x)$ for $x\in Z(F)$ are pairwise disjoint and disjoint from $V(F)\setminus Z(F)$, since $F'$ is a proper copy of $T'$ in $D'$, and by our choice of $z$. It thus remains to verify that every vertex $x\in V(T)$ with $h(x)=0$ is embedded into $R$. Since $F'$ is a proper copy of $T'$ in $D'$ and since $R'\subseteq R$, this is indeed the case for every $x\in V(T')=V(T)\setminus \{l\}$. It then remains to verify that $h(l)=0$ implies that $z\in R'$. Note that by definition of a height function and since $P_1, P_2$ are directed paths starting in $s$ in $T$, we have that $h(l)=h(v)-\ell_1+\ell_2=\ell_2-\ell_1$. Hence, $h(l)=0$ implies that $\ell_1=\ell_2$. Recall that we established above that every directed path in $D'$ starting in $\iota(s)$ of length $\ell_1$ ends in $R'\subseteq R$. In particular, this is the case for the path $Q_2\cup\{(\iota(u),z)\}$ in $D'$, which is of length $|A(Q_2)|+1=(\ell_2-1)+1=\ell_2=\ell_1$. Hence, we indeed have $z\in R$ if $h(l)=0$, as desired. This shows that $F$ is indeed a proper copy of $T$ in $D$, as desired.

As we found the desired contradiction in each of the three cases, this means that our initial assumption on the existence of $T$ was correct. This shows that the statement of the lemma holds, concluding the proof.
\end{proof}

Finally, we can easily complete the proof of our main result, Theorem~\ref{thm:main} by combining Observation~\ref{obs:obvious} and Lemma~\ref{lem:keyfortheorem}.
\begin{proof}[Proof of Theorem~\ref{thm:main}]
By Observation~\ref{obs:obvious}, it suffices to show that every max-grounded tree is $\delta^+$-enforcible. So let $T$ be any given max-grounded tree, let $k:=|V(T)|$ and define $d=d(T):=10^{13k^3(8k)^k}$. Let $D$ be any digraph satisfying $\delta^+(D)\ge d$. By keeping exactly $d$ out-arcs at each vertex, we may w.l.o.g. assume $d_D^+(v)=d$ for each $v\in V(D)$. By the comment after Definition~\ref{def:broomdigraph} we have that $D$ is a $(k,d)$-broom digraph. Additionally, one can observe that $D$ meets the conditions of Lemma~\ref{lem:keyfortheorem}. Hence, by the lemma there exists a (proper) copy of $T$ in $D$. Since $D$ was chosen arbitrarily, this by definition proves that $T$ is $\delta^+$-enforcible, as desired. This concludes the proof of the theorem.
\end{proof}
\section{Concluding remarks}
Let $d_k$ denote the minimum $d$ such that every digraph with minimum out-degree at least $d$ contains every grounded tree on $k$ vertices. Our result shows that $d_k$ exists and $d_k\leq 2^{(Ck)^{k}}$ for some large enough constant $C$. It would be interesting to understand the dependence of $d_k$ on $k$ better. For example, is it polynomial, exponential or double exponential? 
\bibliographystyle{abbrv}

\end{document}